\documentclass[12pt]{amsart}
\usepackage{url,graphicx}
\usepackage{subeqnarray}
\usepackage{mathrsfs,dsfont}

\newcommand{\pa}{\partial}
\newcommand{\ep}{\varepsilon}

\newcommand{\La}{\Lambda}

\newcommand{\vp}{\varphi}
\newcommand{\SR}{Special Relativity}

\newcommand{\RR}{{\mathbb{R}}}

\newcommand{\cale}{\mathscr E}
\newcommand{\calm}{\mathscr M}

\newcommand{\vecteur}[1]{\overrightarrow{#1}}

\newcommand{\ee}{\mathbf e}
\newcommand{\uu}{\mathbf u}
\newcommand{\nn}{\mathbf n}
\newcommand{\bb}{\mathbf b}
\newcommand{\cc}{\mathbf c}
\newcommand{\BB}{\mathbf B}
\newcommand{\EE}{\mathbf E}
\newcommand{\FF}{\mathbf F}

\newtheorem{Theorem}{Theorem}
\newtheorem{Definition}{{Definition}}
\newtheorem{Remark}{Remark}

\begin{document}

\title{The Relativistic Rotation Transformation and the Observer Manifold}
\address{Laboratoire de Math\'ematiques de Reims (CNRS, UMR9008), Universit\'e de Reims Champagne-Ardenne, Moulin de la Housse, B.P. 1039, F-51687 Reims Cedex 2, France and
GREI (EPHE/PSL and Sorbonne-Universit\'e, Paris), 14 Cours des Humanit\'es, F-93322 Aubervilliers {Cedex}, France}
\author{Satyanad Kichenassamy}\email{satyanad.kichenassamy@univ-reims.fr}

\maketitle

\markboth{S. Kichenassamy, \emph{Axioms} 2023, 12(12), 1066}{S. Kichenassamy, \emph{Axioms} 2023, 12(12), 1066}

\bigskip
\textbf{Appeared in:} \emph{Axioms} 2023, 12(12),1066 \\ (\url{https://doi.org/10.3390/axioms12121066})
\medskip

\begin{quote}\scriptsize
\textbf{Abstract. }We show that relativistic rotation transformations represent transfer maps between the laboratory system and a local observer on an observer manifold, rather than an event manifold, in the spirit of C-equivalence. Rotation is, therefore, not a parameterised motion on a background space or spacetime, but is determined by a particular sequence of tetrads related by specific special Lorentz transformations or boosts. Because such Lorentz boosts do not form a group, these tetrads represent distinct observers that cannot put together their local descriptions into a manifold in the usual sense. The choice of observer manifold depends on the dynamical situation under consideration, and is not solely determined by the kinematics. Three examples are given: Franklin's rotation transformation for uniform plane rotation, the Thomas precession of a vector attached to an electron, and the motion of a charged particle in an electromagnetic field. In each case, at each point of its trajectory, there is a distinguished tetrad and a special Lorentz transformation that maps Minkowski space to the spacetime of the local observer on the curve.

\end{quote}

\textbf{Keywords: } {axiomatics of manifolds; observer manifold; Lorentz transformation; rotation; Special Relativity; General Relativity; charged particles; Thomas precession}

\textbf{PACS:}{04.20.Cv; 04.80.Cc; 03.30.+p; 52.20.Dq}

\section{Introduction}

This paper continues a recent work~\cite{SK-23-observer}, which built upon an elaborate discussion of the equivalence principle in the 1960s, initiated in the Princeton school; see
~\cite{K-cr-64a,K-ihp-ceq,K-ihp-ceq-n,C-eq20}. We showed that a proper axiomatisation of these considerations, in the framework of Hilbert's Sixth Problem, required an extension of the manifold concept. One should mathematically distinguish the events that observers measure and the setups that each observer uses to actually perform measurement. The setup of an observer is encapsulated in the metric of a Lorentzian manifold. We are interested here in simple situations, which all involve uniform rotation, in which this identification of events and observers is not appropriate. In all these cases, the event manifold is a Minkowski space, which may be thought of as the system attached to a laboratory using a standard measuring apparatus. The event manifold is, therefore, associated with a distinguished observer: the one actually recording the positions of various particles, each of which has its own system. By contrast, other observers will be attached to different points of the trajectories of particles in motion; it is not feasible to assume that these observers are attached to a measurement apparatus, let alone a human~observer.

While the successes of General Relativity made it reasonable to give the set of events the structure of a Lorentzian manifold on which Einstein's field equations are valid, this event manifold is constructed by a single observer or, more precisely, by a class of observers whose observations are coordinated. Therefore, if we postulate that all observers make measurements in the same way, \emph{{we should assign to every observer a different manifold}}.
 The equivalence of these manifolds, which holds in \SR, has no reason to be valid in more-general situations. The observer manifold is the collection of all these manifolds, together with axioms specifying their relation to the event manifold. These observers could also be human observers with a measuring apparatus, but this is an unrealistic assumption in general.

Therefore, the observer manifold is not a manifold in the usual sense, but may be thought of as a collection of manifolds, each of which represents the current state of the system dragged by the particle under consideration, along its path. This path is observed in the laboratory system as a sequence of events, which trace out a trajectory. In Special Relativity, inertial observers \emph{can} be identified to one another and may be viewed as embedded in a Minkowski space (\cite{SK-23-observer}, Theorem 1). However, in other cases, it is necessary to consciously distinguish the event and observer manifolds. We are interested here in three fairly explicit examples of observer manifolds that may not be identified with the event manifold. All the examples in this paper are related to systems or particles in rotation, since this is one of the simplest sources of non-inertial behaviour. These three examples, therefore, provide the beginnings of a mathematical classification of observer manifolds.

The difficulty of the problem of rotation in Relativity stems from the inadequacy of what would be the natural approach in Newtonian Mechanics, namely to try and generalise a transformation, in cylindrical coordinates $(r,\theta,z,t)$, of the form $r'=r$, $\theta'=\theta-\omega t$, $z'=z$, where $\omega$ is a constant. Even if it were modified and supplemented by a suitable transformation of the time variable, such as the one proposed in 1922 by Franklin~\cite{franklin} (see~\eqref{eq:RRT} below) or others~\cite{langevin-21,langevin-37,post}, this would merely yield a change of coordinates. Now, in Relativity, coordinate changes have no physical meaning; this is the basis of the principle of general covariance~\cite{kretschmann-17,anderson-chiu}. This is similar to the case of surface integrals, in which reparameterisation of the surface does not change the value of the integral. Covariance is not a symmetry, and changes of coordinates are not symmetry transformations, since a Riemannian manifold may have a trivial group of isometries. There are very few spaces, such as Euclidean space, in which one may introduce distinguished coordinates compatible with a global symmetry~group.

We begin with background information, reviewing the axiomatics of the event manifold and of the observer manifold (Section~\ref{sec2}). Special cases include manifolds in the usual sense, including fibre bundles, but are not reducible to them. The event manifold may be viewed as the system attached to a ``laboratory'', which may perform measurements according to a ``standard'' apparatus and records events on a Lorentzian manifold. On the other hand, the system attached to an elementary particle will be associated with a sequence of observers, one for each position of the particle. We then review (Section~\ref{sec3}) the Newtonian treatment of rotation and the difficulties of a straightforward generalisation of it.
Section~\ref{sec4} deals with the reinterpretation of the relativistic rotation transformation (RRT) introduced by Franklin and rediscovered many times. Here, the observer manifold is obtained by attaching a different copy of Minkowski space to every point of every circular trajectory. However, the radius $r$ of the trajectory plays the role of a parameter and not a coordinate. That is why the RRT is not a change of coordinates in Minkowski space. The applications of Franklin's rotation transformation have already been discussed at length in the literature~\cite{K:aflb:96b}. We, therefore, focus on its impact on the definition of the observer manifold in the context of rotation.

Section~\ref{sec:thomas} deals with the motion of a vector attached to a moving electron, with an application to Thomas precession, and Section~\ref{sec6} with the motion of charged particles in a constant electromagnetic field. A few results that also hold for variable fields are given in Section~\ref{sec7}. In both of these examples, the observer manifold is again a collection of spaces attached to every point of a trajectory. Concluding remarks and perspectives are given in Section~\ref{sec8}.

The main point of this paper is that rotation in Relativity cannot be represented as a parameterised sequence of motions in a pre-existing space, but should be viewed as a sequence of manifolds representing local systems of reference, which form together the observer manifold.

\section{The Evolution of the Manifold Concept as It Was Applied to Physics: Historical~Perspective}\label{sec2}

The notion of a manifold went through four stages before it was applied to Relativity. It was first developed for surfaces in ordinary space by Gauss, then extended to three-dimensional manifolds by Riemann, without reference to an ambient space. In parallel, it was realised by Felix Klein that Euclidean geometry and similar ones admitting spaces of motions could be conveniently studied by direct consideration of their groups of (global) symmetry. Fourth, the development of tensor calculus showed that one could define geometric quantities and objects such as tensors in manifolds without any symmetry. The introduction of Relativity showed that one needs to axiomatise the way different observers compute geometric quantities on the basis of a measurement protocol, which was based on ``standard rods and clocks''. As a result, different observers do not have direct access to each others' measurement apparatus. We briefly review these stages and write out the different axiomatics that were developed in this connection.

For a general history of the manifold concept in mathematics, see \cite{scholz}. We limited ourselves to those points relevant to our considerations.

\subsection{Local Differential Geometry and Global Group-Theoretic Models}

The introduction of spherical geometry in ancient Astronomy could have paved the way for a geometry of curved surfaces. It did not, because all constructions on the sphere could also be viewed as taking place in a Euclidean space of dimension three. This way of considering space remained standard until the Nineteenth Century. Now, in Euclidean space, there is a distinguished class of coordinates that have metrical meaning; in rectangular coordinates $(x,y,z)$, for example, the difference in the $z$ coordinates of two points with the same $x$ and $y$ coordinates directly gives their distance. These coordinates are not unique, but there is a limited set of transformations, namely Euclidean motions, that relate to one another all possible such coordinates with metrical meaning, within what was viewed as absolute space.

Gauss first showed how to work with coordinates that do not have direct geometric meaning. In his analysis of surveying, he considered a surface in Euclidean three-space, in which points are labelled by arbitrary coordinates $(u,v)$. He showed that one could define a measure of curvature that does not depend on coordinates, namely what we now call Gaussian curvature. Gauss also defined geometric quantities that depend on the derivatives of the unit normal to the surface, which define the ``second fundamental form''. But, the main progress for our purposes was the introduction of line measurement in arbitrary coordinates and the search for quantities that do not depend on coordinates.

Riemann generalised Gauss' considerations to three-dimensional space and, thus, did away with the need to consider an ambient space since four-dimensional space had not been considered yet. He showed that, if we have a Riemannian metric in coordinates $(u,v,w)$ and assume all metric components are smooth, one can define a quantity, which we now call the Riemann tensor, that must vanish if we are in Euclidean space. One can then make a change of coordinates to rectangular coordinates $(x,y,z)$, at least locally. Therefore, here too, space is absolute, and there exists a distinguished class of coordinates. His work, as continued by Christoffel, Ricci, Levi-Civit\`a, and \'Elie Cartan, also showed how to work in spaces in which the Riemann tensor does not vanish, so that there are no such distinguished coordinates and, in general, no symmetries at all, as well as no reduction, even in a limited region, to Euclidean space. Indeed, it is not possible to make the Riemann--Christoffel connection coefficients $\Gamma^a_{bc}$ vanish in a full neighbourhood of a point; otherwise, the Riemann tensor would vanish as well and space would be flat in this neighbourhood. This will have the consequence that, in General Relativity, even in free fall, the effects of gravity are measurable at second order. Therefore, when one tried to generalise inertial systems in General Relativity, it became clear that they could not be identified with small patches of an event manifold.

Felix Klein took a different view. He showed in his Erlangen Program~\cite{klein-erlangen} that Euclidean geometry is best defined not by a class of axioms on lines and points, but by the existence of a group of transformations that globally transform space into itself, while mapping one set of rectangular coordinates or, equivalently, an orthonormal triad with a base point to another. Non-Euclidean geometries were simply the geometries deduced from different groups. The various models of these new geometries are merely different ways of realising a set of points admitting a family of maps that generalise Euclidean motions. This gives prominence to the group concept, which, however, will turn out to be too restrictive. Indeed, recall the definition of a group, making it slightly more explicit than is usually done.
\begin{Definition}[\textbf{Definition of a group}]
A set $G$ is called a group if
\begin{enumerate}
 \item[(Op)] (Operations are always defined.) For any two $g$ and $g'$ in $G$, there is an element of $G$ called $gg'$, called the composition, or product of $g$ and $g'$ (in $G$).
 \item[(Id)] (Existence of an identity.) There is an element $e\in G$ such that $eg=ge=g$ for every $g\in G$.
 \item[(Inv)] (Existence of an inverse.) For every $g\in G$, there is an element in $G$, called $g^{-1}$, such that $gg^{-1}=g^{-1}g=e$.
 \item[(As)] (Associativity of group law.) For every $(g,g',g'')\in G^3$, we have $(gg')g''=g(g'g")$.
\end{enumerate}
If $X$ is any set, one says that a group $G$ \emph{acts on} $X$ if the following three axioms hold:
\begin{enumerate}
 \item[(Act)] (Action is always defined.) To any $g$ in $G$, one associates a mapping $F_g : X\to X$.
 \item[(Id)] (Effect of the identity.) For any $(g,x)\in G\times X$, $F_e(x)=x$.
 \item[(Com)] (Compatibility with group law.) For every $(g,g',x)\in G^2\times X$, we have $F_{gg'}(x)=F_g(F_{g'}(x))$.
\end{enumerate}
\end{Definition}
One key point that will fail in the sequel is the assumption that group laws or actions are always defined. Two examples of transformations that do not form a group will be met repeatedly: transfer maps between observers and special Lorentz transformations. Intuitively, while it makes sense to assume that all observers may perform the same types of measurements on a similar apparatus, there is no reason to assume that they have full access to each other's measurements, unless we consider only observers that all coordinate their protocols with a centralised entity.

There is another feature of Euclidean space that will have to be relinquished: that it should be possible to establish a correspondence between remote regions of space. Euclidean motions not only map points to points, but also vectors to vectors. Indeed, to two points $A$ and $B$, one associates a vector $\vecteur{AB}$, and if $C$ is another point, there is a fourth point $D$ such that $\vecteur{AB}=\vecteur{CD}$. One speaks of equipollence to represent this equality of vectors. But, two points on a curved surface generally do not define a vector that would lie on the surface in any reasonable sense. Parallel transport was introduced to partially alleviate this difficulty, by replacing $\vecteur{AB}$ by a vector tangent to the surface at $A$, but parallel transport requires a curve and a rule for transport: the equipollence law cannot be determined by the points $A$ and $C$ alone: one needs to specify a geodesic.

All these difficulties explain that we need to introduce a set of axioms that does not implicitly assume the existence of distinguished coordinates or global correspondences between the representations of observers.

\subsection{Axiomatics of the Event Manifold}

When Einstein put forth Special Relativity in 1905, and throughout his elaboration of General Relativity, it was apparent that he was interested in what observers could measure locally. There is no operational way to compare instantly quantities at different points. However,
Minkowski pointed out that Einstein's 1905 representation of spacetime could be viewed as the introduction of an indefinite scalar product in a (global) four-dimensional space, leading to what we now call Minkowski space $M_4$. Therefore, Special Relativity could be analysed in two ways. The first is to say, with Einstein, that inertial observers are each associated with a different Minkowski space, which has only local validity. Lorentz transformations then express the relations between the representations of the same events by different observers. Each observer seeks to identify a local inertial frame in which distinguished coordinates having direct physical meaning may be introduced. The second view is to assume that the set of events is itself endowed with the structure of a global Minkowski space, of which parts are detected by individual observers. As we have shown (\cite{SK-23-observer}, Theorem 1), this amounts to identifying the spacetimes of all observers to a single one. While this is meaningful in Special Relativity, this is not the case in General Relativity.

There is a further difficulty. The interpretation of measurements, such as that of the meson's lifetime, requires one to assign to an elementary particle with nonzero rest mass, such as a meson, a local inertial observer that travels with it. That is how we account for the fact that the lifetime of a meson in its inertial system differs from its lifetime in the laboratory system in which its path is being tracked. Now, no \emph{human} observer exists that would, so to speak, ``ride along an elementary particle'' and could perform the operations allowed by the postulates of Special Relativity. The observers of the theory are, therefore, \emph{notional observers}, rather than actual observers: the operations of coordinate assignment and coordinate change are never actually carried out, except possibly for a ``laboratory system''. When several observers in relative uniform motion meet at a spacetime point, their notional observers are related by Lorentz transformation, so that there is for them only one Minkowski space in which all their translational motions take place, but this identification of notional observers has no reason to hold for non-uniform motions. It is possible to circumvent this difficulty by postulating instead that the de Broglie wave of a particle includes, through its dynamics, the elements of a local system attached to it~\cite{C-eq20}. This wave is a purely material object that is represented in the theory by the notional observer associated with the particle. The observers in the sequel are always understood in this sense; they do not require human intervention and, indeed, such intervention is generally not feasible.

In his search for a relativistic theory of gravity, Einstein took over the existing theory of Riemannian manifolds and tried to introduce into it a physical content adapted to his vision of the gravitational field. We currently read his attempts as leading to the following axiomatics of the event manifold: events accessible to measurement for all observers, inertial or not, are points on an event manifold with a Lorentzian metric, satisfying Einstein's field equations; for background material, see (\cite{SK-nlw}, Chapter 6). The standard definition of such a manifold is as follows.
\begin{Definition}[\textbf{Definition of a manifold of dimension $n$}]\label{def:eventmanifold}
A (differential) manifold $M$ of dimension $n$ and of class $C^k$, where $k\geq 1$ or $k=\infty$, is a set equipped with a family $(U_\alpha, \vp_\alpha)_{\alpha\in Z}$, called an \emph{atlas}, indexed by an arbitrary set $Z$ such that:
\begin{enumerate}
 \item[(Cov)] (Covering axiom.) For every $\alpha\in Z$, the set $U_\alpha$ is a subset of $M$. Furthermore, any point of $M$ belongs to at least one $U_\alpha$.
 \item[(CA)] (Coordinate axiom.) Each $\vp_\alpha(U_\alpha)$ is an open subset of $\RR^n$, the latter being endowed with its usual topology, and the $\vp_\alpha$ are one-to-one. Therefore, each point $P$ of $U_\alpha$ may be written $P=\vp_\alpha^{-1}(x^1,\dots,x^n)$ for some $\alpha$, where $(x^1,\dots,x^n)\in\RR^n$ is uniquely determined by $\alpha$. These $n$ numbers are called the \emph{local coordinates} of $P$ in the \emph{coordinate chart} $(U_\alpha, \vp_\alpha)$.
 \item[(CC)] (Axiom on coordinate changes.) If two coordinate charts $(U_\alpha, \vp_\alpha)$ and $(U_\beta, \vp_\beta)$ are such that $U_\alpha\cap U_\beta\neq\emptyset$, the map $\vp_\beta\circ\vp_\alpha^{-1}$ is of class $C^k$ on its domain of definition, as so is its inverse $\vp_\alpha\circ\vp_\beta^{-1}$.
\end{enumerate}
One defines a topology on $M$ by saying that a set $X\subset M$ is open if, for any coordinate chart $(U_\alpha, \vp_\alpha)$, the sets $\vp_\alpha(X\cap U_\alpha)$ are open in $\RR^n$. Submanifolds are defined in the obvious way.
\end{Definition}
By construction each map, $\vp_\alpha$ is a homeomorphism between $U_\alpha$ and $\vp_\alpha(U_\alpha)$. We shall deal throughout with smooth manifolds ($k=\infty$) for simplicity. The definitions of geometric objects, including tensors and connections, are standard~\cite{hawking-ellis,k-n}, including more-general ``geometric objects'' such as Lie derivatives (\cite{yano} p.~18 sq.). In particular, a Lorentzian manifold is a manifold endowed with a metric tensor with signature $(3,1)$ (with one negative square). As a consequence, coordinates lose all physical meaning. Geometric objects on $M$ are determined by their expressions in the different coordinate patches $U_\alpha$, related by specific transformation laws, which form the basis of general covariance. Thus, general covariance is not a symmetry; it is similar to a change of variables in an integral, that is simply a change of label, that does not affect the value of the integral.

This provides a convenient framework for the event manifold $\cale$ of a fixed observer, which we may call the ``laboratory'' observer. Each point represents a spacetime coincidence, and the metric enables this particular observer to derive coordinate-independent information from such events. The question is the status of other observers. Einstein's 1905 operational procedure, still in use, refers to two observers having access to ``standard'' rods and clocks. Now, the predictions of Special Relativity are also found to be correct in situations where such a procedure cannot be carried out, as in the case of the rest system of an elementary particle. The accuracy of the predictions of Special Relativity means that there must be a material structure that performs the same function as Einstein's human observers, as well as a mathematical structure to represent it. The de Broglie wave could serve this purpose~\cite{C-eq20}. As for the mathematical structure, it must be similar to the system of the laboratory, since it is a basic axiom that all observers perform the same types of operations. Every observer $A$ is, therefore, associated with a different manifold $S_A$ with an origin $O_A$. There is no reason to assume that $S_A$ extends beyond a small neighbourhood of $O_A$; in all the examples of this paper, $S_A$ consists of part of a four-dimensional vector space. The origin $O_A$ is tracked by the laboratory $L$. Therefore, there is a curve $X_A$ on the event manifold $\cale$ that represents the positions of $O_A$ in $\cale$. In addition, for events that are close to $X_A$, $A$ and $L$ have possibly different representations of them. We call \emph{transfer maps} the mappings that connect these representations. Again, they are locally defined, possibly even only infinitesimally. Also, observer $A$ is not fully determined by its trajectory---its trace $X_A$, in the sense of axiom (T) in Definition \ref{def:observermanifold1} below---on the event manifold: if another observer $B$ happens to cross the trajectory of $A$, the two observers will remain different, except if they exactly have the same state of motion. As a result, it is not possible to define a topology on $\cale$ in such a way that one can separate $X_A$ and $X_B$ by non-intersecting neighbourhoods; these are the same point in $\cale$. If we agree that observers are not events, we, therefore, need fairly general sets of axioms that should encompass all these possibilities. We give two such sets below.

\subsection{First Set of Axioms for the Observer Manifold}

Consider an observer $A$ of which the origin traces a curve $X_A$ on the event manifold $\cale$, parameterised by proper time $s$. Thus, in a local chart of $\cale$, $X_A : s\mapsto (x^a(s))$. The observer labels events on a different manifold $S_A$, and the observer manifold is the collection of all these manifolds $S_A$. The event manifold also represents the events as recorded by another observer, which we called the ``laboratory'' $L$, even though it, too, may not be human. The issue, therefore, is to understand how $A$ and $L$ compare their descriptions. In Special Relativity, the comparison of these descriptions leads to a complete identification of their representations~\cite{SK-23-observer}, up to a Lorentz transformation, if we assume that they compare their descriptions via local diffeomorphisms, as in the following definition, taken from that~paper.
\begin{Definition}\label{def:observermanifold1}
An \emph{event manifold} is a (smooth, differential) manifold $\cale$ of dimension four and of class $C^\infty$. An \emph{observer manifold} $\mathscr M$, over the event manifold $\cale$, is a family $(S_A)_{A\in \mathscr A}$ of manifolds in the sense of Definition~\ref{def:eventmanifold}, indexed by an arbitrary set $\mathscr A$, such that:
\begin{enumerate}
 \item[(M)] (Manifold associated with an observer $A$.) Each $S_A$ is a Lorentzian manifold (typically, consisting of only one chart) with a distinguished origin $O_A\in S_A$. It represents the events as recorded by $A$.
 \item[(T)] (Trace of an observer on the event manifold.) Each $S_A$ has a trace $X_A$ on the event manifold $\cale$.
 \item[(TE)] (Transfer map to the event manifold.) If $A$ is an observer, then there is an open set $V_A$ in $S_A$, related by a diffeomorphism $\phi_{A}$ to an open set in $\cale$. It is not an isometry in general.
 \item[(TO)] (Transfer map between observers.) If two observers $A$ and $B$ can represent some of the same events ($W:=\phi_{A}(V_A)\cap \phi_{A}(V_B)\neq\emptyset$), then the open sets $U_A:=\phi_{A}^{-1}(W)$ and $U_B:=\phi_{B}^{-1}(W)$, in $S_A$ and $S_B$, respectively, are related by a diffeomorphism $\phi_{AB}$, with $\phi_{AB}(O_A)=O_B$.
\end{enumerate}
The metric on each $S_A$ is not assumed to be the pull-back of the metric on $\cale$ by $\phi_A$.
\end{Definition}
One could make the definition slightly shorter by identifying the event manifold with the system of the (distinguished) laboratory that records the events, but it may be best to be more explicit. Indeed, the observer attached to a state of motion of an elementary particle is fundamentally different from the laboratory observer: there are neither human observers nor actual measuring instruments riding on it~\cite{C-eq20}.

The proviso that the maps $\phi_A$ are not necessarily isometries makes this definition irreducible to the definition of a manifold in the ordinary sense. Otherwise, one could have viewed the open sets $V_A$ as coordinate patches on the event manifold. The observer manifold is not a fibre bundle either, in which the $S_A$ would be fibres, because every point on the base of a fibre bundle should have a well-defined fibre. Here, fibres are defined only over points of the traces $X_A$ of actual material particles. Moreover, if two observers $A$ and $B$ happen to meet at the same event $X$, there would be two of these supposed fibres over it, namely $S_A$ and $S_B$. Another approach to dynamics is to make the metric of a spacetime dependent on the tangent to the worldline of an observer, as in Finsler geometry. Recent authors have suggested similar approaches on a flat background~\cite{friedman-action}; see also~\cite{scarr-friedman-action} (we have only seen the former work, the latter one having been pointed out by a referee). Again, these developments cannot represent several observers meeting at one point.

For all these reasons, it is necessary to introduce axiomatics that are as inclusive as possible, to avoid unwittingly ruling out logical possibilities. The transfer maps between observers, like the coordinate changes in an atlas, do not form a group because their composition is not always defined. Therefore, Klein's group-theoretic approach cannot be generalised either, because coordinate changes are not symmetries.

This definition is still not general enough. It is adequate in the case observers are in a position to record events in a full neighbourhood of their origin, as would be the case for human observers, or devices that include identical physical measuring apparatuses. However, the apparatus, being itself material, is affected by the local gravitational field. This will lead to the C-equivalence principle, which needs to be incorporated into the axiomatics.

\subsection{Conformal-Equivalence as a Mathematical Expression of Einstein's View of His Equivalence Principle}

Einstein's development of General Relativity as recorded in his own account~\cite{einstein:meaning} may be viewed in the light of later developments as an attempt to cast his physical views into an existing mathematical framework. We suggested that the mathematical concept of manifolds as developed before Einstein is not sufficient to accurately represent his view---or what has been ascertained in later experiments as well. Summarising our earlier analysis~(\cite{SK-23-observer}, Section 2), the main issue is to understand how the notion of an inertial system should go over to General Relativity. Inertial systems in \SR\ are represented by copies of Minkowski space, with a distinguished time-oriented vector representing the (tangent to the) world-line of the given observer. The fact that rotation may be detected by local experiments, such as Sagnac's, shows that each observer, even non-inertial, should be able to define a local inertial system and to determine geometric quantities by the contraction of tensors with the vectors of a local orthonormal tetrad (ONT).

The question now is: What is the mathematical object that corresponds to this local inertial system? The simplest answer would be to take an inertial system to be a small patch on the event manifold $\cale$. But, this is impossible, because it would mean that one could reduce the metric on $\cale$ to the Minkowski metric. Mathematically, this is not possible unless the curvature tensor vanishes. However, the metric on $\cale$ can be reduced to a sum of squares \emph{at one point}. This suggests identifying the local inertial system to the tangent space of $\cale$ at one event. Because of the assumption that all inertial observers are equivalent, this tangent space should be identified with the Minkowski space of a standard inertial observer. But, even this is not sufficient because the Pound--Rebka experiment showed that the ``standard rods and clocks'' posited by the theory are affected by the local gravitational field. Therefore, if the interpretation of measurements is to give the observed results, we need to posit that local non-rotating systems are not identical to Minkowski space, but only conformally equivalent to it. The isotropy of local descriptions~(\cite{SK-23-observer},~Sections 2.3 and 2.4) and the constancy of the speed of light leaves the choice of a conformal factor as the only freedom. We summarise this as follows.
\begin{quote}
All local observers are provided with standard measuring devices (identical devices having the same behaviour at the same point when relatively at rest). \emph{Because of the presence of the field, these devices have a behaviour that varies from point to point}. This may be expressed by the proportionality of the metrics on the manifold $S_A$ associated with observer $A$ and the metric of an inertial observer:
\begin{equation}\label{fg:46.1}
 (ds^2)_A=\La_A\overline{ds}^2,
\end{equation}
where the elementary squared interval between two events is $(ds^2)_A$ according to an observer located at $A$, but would be $\overline{ds}^2$ according to an inertial observer. The quantity $\La_A$ is not accessible to experiment, since $\overline{ds}^2$ cannot be measured when a nonzero field is present, but the \emph{ratio} $\La_A/\La_B$ comparing the deviations from \SR\ at two different points $A$ and $B$ is accessible to measurement~\cite{K-ihp-ceq,K-ihp-ceq-n}.
\end{quote}
The local system of reference will be called a \emph{pseudo-inertial system}.

We shall say that a theory is \emph{compatible with C-equivalence} if it is consistent with:
\begin{enumerate}
 \item Weak equivalence;
 \item The identity of local descriptions of identical phenomena;
 \item The isotropy of local spacetime;
 \item The pseudo-inertial character of the local system of reference.
\end{enumerate}
The letter C stands for ``conformal'' since pseudo-inertial systems are merely conformally Minkowskian and not Minkowskian. All the usual verifications of General Relativity can be accounted for in this framework and, indeed, require it since setting the conformal factor to unity leads to results incompatible with measurements~\cite{K-ihp-ceq,K-ihp-ceq-n,K-necessaire-66}.

We now turn to a set of axioms for the observer manifold compatible with C-equivalence.

\subsection{Second Set of Axioms for the Observer Manifold}

We consider now a set of axioms adapted to the trajectory of a material particle. Each point of its trajectory will be assigned a manifold. We, therefore, associate the various states of motion of the particle with a different observer.
\begin{Definition}\label{def:observermanifold2}
An \emph{event manifold} is a (smooth, differential) manifold $\cale$ of dimension four and of class $C^\infty$. Consider a trajectory on $\cale$, parameterised by the arc length (for the metric on $\cale$): $X : I\to\cale, s\mapsto X(s)$, where $I$ is an interval on the real line. The \emph{observer manifold} $\calm$ for this trajectory, over the event manifold $\cale$, is a family $(\calm_s)_{s\in I}$ of manifolds in the sense of Definition~\ref{def:eventmanifold}, indexed by $s$, such that:
\begin{enumerate}
 \item[(M)] (Manifold associated with $s$.) Each $\calm_s$ is a Lorentzian manifold (typically, consisting of only one chart) with a distinguished origin $O_s\in \calm_s$.
 \item[(TE')] (Transfer map to the event manifold.) For every $s$, there is, for every value of the arc length parameter $s$, a map $\psi_{s}$ from $\calm_s$ to the tangent space $T_{X(s)}\cale$. Here, $\psi_{s}(O_s)$ is the origin in $T_{X(s)}\cale$.
 \item[(CM)] (Conformally Minkowskian metric on $\calm_s$.) For every $s$, the metric on $\calm_s$ differs from that of an inertial system by a conformal factor $\Lambda_s$.
\end{enumerate}
\end{Definition}
The second definition is wider than the first, because it only requires the specification of maps from the observers to the laboratory. Just as before, even if we consider a particle in motion that carries a distinguished frame, it is not possible to view the observer manifold as a frame bundle, because there may be different fibres at the same point and because the transfer maps are not necessarily projections onto the base. Therefore, observers in different states of motion are represented by different mathematical objects, even if they pass through the same point on the event manifold. This requires the introduction of a mathematical object of which the points are the observers, as was performed above.

In intuitive terms, each observer is now identified with a copy of the tangent space of the event manifold, endowed with a metric that depends on the state of the observer. The trajectory $X(s)$ in Definition \ref{def:observermanifold2} corresponds to the trace of an observer $A$ in Definition \ref{def:observermanifold1}. The main new element is that different observers are introduced for different points on the trajectory. Thus, the observer manifold contains observers that only lie along the trajectory of $X_A$. Typically, the conformal factor depends on the expression $g_{ab}u^au^b$ where $u^a$ is the unit tangent to the trajectory of $X_I$ in $\cale$; see~\cite{K-ihp-ceq}. One could further generalise this set of axioms by requiring that the metrics on $\cale$ and $S_A$ agree to any given order $m$; Definition~\ref{def:observermanifold2}, then, corresponds to $m=1$. But, this will not be needed for the applications we have in view here. The introduction of several non-intersecting trajectories is straightforward, since each trajectory may be handled independently. In the case of intersecting trajectories, a new axiom similar to TO in Definition \ref{def:observermanifold1} could be introduced, but will not be needed in this paper.

We now turn to examples related to uniform rotation. In all these cases, Definition~\ref{def:observermanifold2} will be the appropriate one, and the freedom to associate different observers with different points on the trajectory will be essential. Since no gravitational field is present, the conformal factor will be equal to unity. However, the other requirements of C-equivalence will still hold. We begin by recalling the difficulties of representing uniform rotation in Relativity by generalising the notion of rigid rotation. We, then, show that the relativistic rotation transformation introduced by Franklin, and found again by Trocheris and, later, by Takeno, actually represents transfer maps in the sense of the above definition and not diffeomorphisms or changes of coordinates. In particular, the radial variable in these transformation is a parameter and not a variable. We, then, apply similar considerations to Thomas precession and to the motion of a charged particle.

\section{Rotation and Rigidity}\label{sec3}

We recall some steps in the analysis of rotation, showing how the view of rotation as a rigid body motion parameterised by a time parameter gradually became untenable, in Newtonian mechanics or Special Relativity.

\subsection{Rotation in Newtonian Spacetime}

Let $N_4 = E_3 \times \RR$ be the Newtonian spacetime. $E_3$ is a
Euclidean space invariant under the six-parameter orthogonal group
SO(3) of translations and rotations, and $\RR$ represents the
translation invariant absolute time; $E_3(t):=E_3\times\{t\}$ is the locus of all
simultaneous events at an arbitrary time $t$.

Motion is, then, conceived of as an application of $E_3(t)$ into
$E_3(t+dt)$, i.e., of $E_3$ into itself:
\begin{itemize}
\item[(a)] Translational motions are time-dependent changes of the spatial origin:
\[
(x^{\alpha}) \mapsto (x^\alpha + a^{\alpha}(t))
\]
in a system $S$ of Cartesian coordinates $x^{\alpha} = (x, y, z)$, where
$\alpha$ or any Greek index takes the values 1, 2, and 3.

\item[(b)]
 Rotational motions leaving the origin fixed are time-dependent
 Euclidean rotations:
\[(x^{\alpha}) \mapsto (x^{\alpha'} = R^{\alpha'}_{\beta}(t) x^{\beta}).
\]
The inverse transformation is given by the inverse matrix of $(R^{\alpha'}_{\beta}(t))$, denoted by $(R^{\alpha}_{\beta'}(t))$:
\begin{equation}\label{eq:rot}
x^{\alpha}=R^{\alpha}_{\beta'}(t) x^{\beta'}.
\end{equation}
\end{itemize}
{Among} these motions, uniform translations, defined by $\dot{a}^{\alpha}(t)
= da^{\alpha}/dt = \text{const.}$,
have the characteristic property of not being
detectable by an experiment internal to the moving reference system $S'$
and define an equivalence class of reference systems, that of inertial
systems, hence the principle of Galilean Relativity consecrating the
invariance of Dynamics under the Galilean group.

Let the frame $S'$ be rotating with respect~to $S$ about their common origin. To determine the relation between the two systems, consider first a
particle $P$ of velocity $\mathbf v$ with respect to $S$. The components of $\mathbf v$ are, therefore, $v^{\alpha} = dx^{\alpha}/dt$. With respect~to $S'$, the same $P$ has a
velocity $\mathbf v'=(v^{\prime\alpha'})$, with components given as follows.
\begin{equation}\label{3}
v^{\prime\alpha'} = dx^{\alpha'}/dt= R^{\alpha'}_{\beta} v^{\beta} +
\dot{R}^{\alpha'}_{\beta} x^{\beta}.
\end{equation}
{Therefore}, using \eqref{eq:rot}, the expression of this vector
 $\mathbf v'$ in $S$, namely $(v^{\prime\alpha})$ (with unprimed index), is given by
\begin{equation}\label{1}
v^{\prime\alpha} =R^{\alpha}_{\gamma'}v^{\prime\gamma'} = v^{\alpha} + R^{\alpha}_{\gamma'} \dot{R}^{\gamma'}_{\beta}
x^{\beta}.
\end{equation}
We now specialise these formulae to the case when particle $P$ is at rest with respect~to $S'$---so that $v^{\prime\alpha}=0$. In that case, $P$ has with respect~to $S$ the velocity:
\begin{equation}\label{2}
 v^{\alpha} = -\omega^{\alpha}_{\beta} x^{\beta} := - R^{\alpha}_{\gamma'
} \dot{R}^{\gamma'}_{\beta} x^{\beta}.
\end{equation}
{This} equation defines the quantities $\omega^{\alpha}_{\beta}$.
Differentiating \eqref{3} with respect to absolute time $t$, we obtain
\begin{equation}\label{4}
\dot{v}^{\prime\alpha'} = \dot{v}^{\alpha'} + 2\omega^{\alpha'}_{\gamma'}
v^{\gamma'} + (\dot{\omega}^{\alpha'}_{\beta'}
+ \omega^{\alpha'}_{\gamma'} \omega^{\gamma'}_{\beta'}) x^{\beta'}.
\end{equation}
{This} is the form of the acceleration of a particle with respect to $S'$.
If we further assume the particle in rotation satisfies $\dot{v}^{\alpha'}= 0$, we see that it has nevertheless a nonzero acceleration in $S'$, which is the sum of the Coriolis
and centrifugal accelerations, respectively given by the second and third terms in
the r.h.s.\ of the above equation.
Their detection by internal effects leads to the so-called ``absolute'' character of rotation or, more precisely, the fact that their existence in a system indicates that this system is not inertial.

Formulae (\ref{3})--(\ref{4}) define the Galilean rotation transformation (GRT).

\subsection{Minkowskian Spacetime and Lorentz Transformations}

Consider Minkowski spacetime $M_4$, namely a four-dimensional homogeneous and isotropic space, with a global system of coordinates $\{x^i\}$, with $i,j,\dots$ taking values $0,\dots,3$.
We first show that rigid motion in Special Relativity cannot be defined satisfactorily and, then, show that
a uniform translational motion is no longer a time-dependent change of
spatial origin, but a hyperbolic rotation represented by the Lorentz boost.
The intuitive reason for these difficulties is that the mapping of the space $E'_{3}$
of a moving observer $S'$ onto the space $E_{3}$ of $S$, unlike the Newtonian case, is not an
element of SO(3) since these two three-spaces are different and are not endowed with distinguished triads, which might enable their identification.

\subsubsection{Difficulties with Rigidity in Special Relativity}

After Ehrenfest (1909)~\cite{ehrenfest-09} remarked that a rigid body cannot be set into
rotation, since the perimeter of the circle described by a point of
the body would suffer Lorentz contraction while its radius remains
unchanged, the compatibility between rigidity and rotation
has been continuously debated. Born \cite{born-09,born-10} generalised the Newtonian concept of a rigid body to that of a time-like rigid congruence such that the corresponding four-velocity
$u^{a}$ obeys the rigidity criterion:
\begin{equation}
P_{ab} = \nabla_{a} u_{b} + \nabla_{b} u_{a} + u_{a} \dot{u}_{b}
+ u_{b} \dot{u}_{a} = 0,
\end{equation}
where
\[
\dot{u}_{a} = u^{c}\nabla_{c}u_{a}.
\]
{Herglotz} (1910)~\cite{herglotz-10} and F. Noether (1910)~\cite{noether-f-10} further investigated this criterion
and showed that those motions have only three degrees of freedom and
that rigid rotation is uniform.
See also Pirani and Williams (1962)~\cite{pirani-williams-62}, Thirring (1993)~\cite{thirring-93}, Trautman (1964)~\cite{trautman-64}, Synge (1965)~\cite{synge-65}, Kichenassamy,
(1982, 1996)~\cite{sk-82-rigid,K:aflb:96b}, and N. Rosen (1947)~\cite{rosen-47}.
von Laue (1921)~\cite{laue-21} gave yet another argument against the existence of the Newtonian rigid
body in Relativity: one can set up an arbitrary number of non-interacting or independent disturbances for a sufficiently short interval of time, such disturbances having finite propagation velocity.

It appears clear that it is impossible to attach a rigid body to a set of particles without contradictions. The only remaining possibility is to define observers using orthonormal tetrads. Indeed, in Newtonian Mechanics, the position of a rigid body is entirely determined by an origin and a triad. However, in Special Relativity, tetrads remain meaningful even though rigid bodies are not available.
To go further, we need to briefly recall some facts about the structure of Lorentz transformations viewed as transformations from one tetrad to another, rather than transformations internal to $M_4$.

Let $(\ee_a)=(e_0,\dots,e_3)$ denote an orthonormal tetrad basis of $M_4$ (ON)T (or tetrad for~short):
\begin{equation}\label{th-1}
 \ee_a=e_a^i\pa_i,
\end{equation}
where $(\pa_i=\pa/\pa x^i)$ is the basis of four-vectors deduced from the coordinates, and the determinant of $(e_a{}^i)$ is assumed to be nonzero. The scalar products of vectors of the tetrad are, therefore, given by
\begin{equation}\label{th-2}
 \ee_a\cdot \ee_b = \eta_{ab}=\delta_{ab}-2\delta_a^0\delta_b^0.
\end{equation}
{The} first vector, $\ee_0=(e_0^i)$, may be viewed as the four-velocity of an observer $O$ or of a particle. As usual, $(e_i{}^a)$ is the inverse (matrix) of $(e_a{}^i)$. The co-basis (or dual-basis) ${\bf\theta}^a$ is defined by
 \begin{equation}\label{th-3}
 {\bf\theta}^a=dx^ie_i{}^a, \text{ and we have }{\bf\theta}^a(\ee_b)=\delta_a^b,\quad
 e_a{}^ie_i{}^b=\delta_a^b;
 \end{equation}
here, the parentheses stand for evaluation (of the one-form ${\bf\theta}^a$ on the vector $\ee_b$). The physical components $(x^a)$ of a vector $x=x^i\pa_i$ are given by the projections of $x$ on the ONT:
\begin{equation}\label{th-4}
 x^a=x^ie_i{}^a.
\end{equation}
{A} change of coordinates $(x^i\mapsto x^{i'}$) does not affect the tetrad (since $\ee_a=e_a{}^i\pa_i=e_a{}^{i'}\pa_{i'}$), so that physical quantities, and hence, the laws of Physics, must be generally covariant, that is have the same tensorial form in every system of coordinates: the physical description cannot depend on how we label events.

Let us now consider linear, non-singular transformations $T$ of $M_4$:

\begin{equation}\label{th-5a}
 T : x\mapsto X\quad \text{ or } \quad x^i=x^ae_a{}^i\mapsto X^i=x^aE_a{}^i,
\end{equation}
where
\begin{equation}\label{th-5b}
 E_a{}^i=e_b{}^iT^b{}_a,\quad\det|T^a{}_b|\neq 0,
\end{equation}
from which we obtain
\begin{equation}\label{th-5c}
e_i{}^a=T^a{}_bE_i{}^b\quad \text{ with } E_a{}^iE_i{}^b=\delta_a{}^b,
\end{equation}
and
\begin{equation}\label{th-6}
X^i=X^a e_a{}^i,\quad X^a=T^a{}_bx^b=x^iE_i{} ^a.
\end{equation}
It follows that $T$ may be viewed:
\begin{enumerate}
 \item [(a)] Either as a basis transformation (\ref{th-5b}): $(\ee_a)\mapsto (\EE_a=\ee_bT^b{}_a)$;
 \item [(b)] Or as a transformation of physical components (\ref{th-6}): $(x^a)\mapsto(X^a=T^a{}_bx^b)$ with respect to the (fixed) basis $(\ee_a)$.
\end{enumerate}
{In} the first case, it relates two conceptually different spacetimes; in the second, it represents a transformation within a single spacetime.

$T$ is called a Lorentz transformation, written $L=(L^b{}_a)$, when it is (linear and) orthogonal, that is preserves the metric:
\begin{equation}\label{th-7a}
 \ee_a\mapsto \EE_a=\ee_bL^b{}_a; \quad \eta_{ab}=\ee_a\cdot \ee_b
 \text{ becomes } \eta_{ab}=\EE_a\cdot \EE_b,
\end{equation}
so that
\begin{equation}\label{th-7b}
 \eta_{ab}=\eta_{cd}L^c{}_aL^d{}_b.
\end{equation}
{Such} transformations form the Lorentz group, a subgroup of the Poincar\'e group, where the latter also includes translations.

Lorentz transformations, under mild smoothness assumptions, are derived from the two physical postulates:
\begin{enumerate}
 \item [(a)] There exists a class of privileged systems of reference, that of inertial frames, with respect to which a free particle moves with a rectilinear and uniform motion;
 \item [(b)] With respect to any inertial frame, light propagates isotropically with a finite, constant velocity.
 \end{enumerate}
{Equation} $X^0=L^0{}_ ax^a$ contains the information that the simultaneity of distant events is not preserved, i.e., the concepts of space and time are not absolute.

Even though Lorentz transformations form a group, there are distinguished transformations between observers, namely special Lorentz transformations, or boosts, that do not form a group~\cite{lalan-53}. We begin with the general properties of boosts (Section~\ref{sec:boosts}) and, then, consider the compositions of boosts (Section~\ref{sec:nogroup}).

\subsubsection{Boosts}\label{sec:boosts} A special Lorentz transformation or boost is the transformation $B(v,u)$ sending the four-velocity (or first frame vector) $u=e_0$ of a frame $F$ to $v=E_0$, the first vector of another frame $\bar F$, while leaving invariant the plane orthogonal to both $u$ and $v$; it is also called a Lorentz transformation without rotation. General Lorentz transformations are obtained from boosts by composing them with spatial rotations.

The boost $B(v,u)$ may be obtained by composing two reflections (or hyperplane symmetries): first, a reflection with respect to the hyperplane orthogonal to $u$, namely $\delta^i{}_j+2u^iu_j$, followed by a reflection with respect to the hyperplane orthogonal to $u+v$, namely
\[ \delta^i{}_j+\frac{(v^i+u^i)(v_j+u_j)}{1+\gamma},
\]
where $\gamma=-v\cdot u$. Carrying out the composition, we obtain the boost
\begin{equation}\label{th-8}
 B^i{}_j=\delta^i{}_j+\frac{(v^i+u^i)(v_j+u_j)}{1+\gamma}-2v^iu_j.
\end{equation}
{We} accordingly introduce the frame $\bar F$ given by $E_a{}^i=B^i{}_je_a{}^j$.

The relative three-velocity $\beta^\lambda=\gamma^{-1}u^\lambda$ of $F$ relative to $\bar F$ is determined by
\begin{equation}\label{th-9}
 u^\lambda=u^i E_i{}^\lambda.
\end{equation}
{We} obtain from Equation~(\ref{th-8}) the $B$-transforms $E_\lambda$ of the space axes $e_\lambda$ of $F$:
\begin{equation}\label{th-10}
 E_\lambda{}^i=
 \left[\delta_\lambda{}^\sigma+(\gamma-1)\hat\beta^\lambda\hat\beta_\sigma
 \right]
 e_\sigma{}^i-\gamma\beta_\lambda e_0{}^i,
 \text{ where } \hat\beta^\lambda=\frac{\beta^\lambda}{|\beta^\sigma|}.
\end{equation}
{Equation}~(\ref{th-10}) shows that the three-space spanned by $(e_\lambda^i{})_{\lambda=1, 2, 3}$ is not transformed by $B$ into itself, but into the space spanned by $(E_\lambda^i{})_{\lambda=1, 2, 3}$; this expresses that the three-space is not absolute, i.e., that the hypersurfaces of simultaneity are not invariant.

From Equation~(\ref{th-8}), we also obtain
\begin{equation}\label{th-11}
 E_\lambda\cdot e_0=-E_0\cdot e_\lambda.
\end{equation}
{This} expresses that the three-velocity $\bar\beta^\lambda$ of $\bar F$ with respect to $F$ (namely $\gamma^{-1}E_0{}^ie_i{}^\lambda$) is the opposite of the three-velocity $\beta^\lambda$ of $F$ with respect to $\bar F$ (namely, $\gamma^{-1}e_0{}^iE_i{}^\lambda$). Only the physical components of the relative velocity with respect to $F$ and $\bar F$ are meaningful.

We now turn to the lack of the group property for boosts.

\subsubsection{Spatial Triads and Lorentz Cycles}\label{sec:nogroup}

A Lorentz cycle $L_n$ is the composition of $n+1$ boosts such that their product leaves unchanged the initial time-like unit vector $u$:
\begin{equation}\label{th-13}
 L_n=B(u,v_n)B(v_n,v_{n-1})\cdots B(v_2,v_1)B(v_1,u).
\end{equation}
{Thus}, $L_n : u\mapsto v_1\mapsto\cdots\mapsto v_n\mapsto u$. We also choose a tetrad $(e_a)$ of which $u$ is the first vector: $e_0=u$. Taking $n=2$ for simplicity and writing $v_1=v$ and $v_2=w$, the cycle $L_2$ transforms the tetrad $(e_a)$ into a tetrad $(G_a)$ defined by
\begin{equation}\label{th-14}
 G_a=B(u,w)B(w,u)B(v,u) e_a.
\end{equation}
{In} particular, since the boosts in the cycle map, respectively, $u=e_0$ to $v$, $v$ to $w$, and $w$ to $u$, we have $G_0=u=e_0$.
Therefore, the spatial basis $(G_\lambda)$ orthogonal to $u$ spans the same three-space as the $(e_\lambda)$; they are given by
\begin{equation}\label{th-15}
 G_\lambda{}^i = (\delta^i{}_j+A^iw_j+B_ iv_ j)e_\lambda{}^ j,
\end{equation}
with
\begin{eqnarray*}
 A^i &=& c^k(v,w)\left[\delta_k{}^i+u_kc^i(u,w)\right], \\
 B^i &=& \left[\delta_k{}^i+c^i(w,v)w_k+c^i(u,v)u_hc^h(w,v)w_k
     \right] c^k(v,u) - c^i(u,v),
\end{eqnarray*}
where $c$ is defined, for any two four-vectors $a$ and $b$, by
\[ c^i(a,b):=\frac{a^i+b^i}{1-a\cdot b}.
\]

At the lowest order of approximation, in which $u\cdot v=v\cdot w=w\cdot u=-1$, Equation~(\ref{th-15}) reduces to
\[ G_\lambda{}^i \approx
  \left[\delta^i{}_j+\frac12(v^iw_j-w^iv_j)-\frac12u^i(w_j-v_j)
  \right]e_\lambda{}^j,
\] or
\begin{equation}\label{th-16}
 G_\lambda{}^i \approx e_\lambda{}^i
 +\frac12(v^\sigma w_\lambda-w^\sigma v_\lambda)e_\sigma^i.
\end{equation}
$(G_\lambda{}^i)$ are obtained by a spatial rotation of $(G_\lambda{}^i)$. Therefore, the modification of space axes when performing a boost is at the origin of the spatial rotation appearing in the kinematical approach.

We have now the general tools at hand to tackle the three situations mentioned in the Introduction. We, therefore, move to the first example of a non-trivial observer manifold.

\section{First Case Study: The Relativistic Rotation Transformation}\label{sec4}


An approach avoiding the pitfalls of the Galilean rotation transformations
 was proposed as early as 1922 by Franklin~\cite{franklin} and rediscovered by
Trocheris~\cite{trocheris} and Takeno~\cite{takeno}. It will,
therefore, be referred to as the FTT approach; its physical meaning
is discussed in~\cite{K:aflb:96b}, together with its applications to models of
pulsars and to the covariance of Maxwell's equations and references to earlier work. Much of the early discussion started at a time when the concept of general covariance~\cite{kretschmann-17,anderson-chiu} had just been introduced, and its consequences had not all been drawn. For this reason, starting with Langevin~\cite{langevin-21}, it was believed that the rotation transformation should be a change of coordinates. We now know that changes of coordinates are changes of labels and have no physical meaning. In particular, what is sometimes called ``Langevin's metric'' is merely Minkowski's metric expressed in a different system of coordinates, and therefore, any interpretation of measurement in which such coordinates play an essential role is necessarily flawed, which was only beginning to be understood. Another difficulty concerns the definition of the angular velocity parameter.
Essentially, in the Galilean approach, the speed $v=\omega r$ and
the angular velocity $\omega$ are both linearly additive; in the
conventional relativistic approach, the speed $v=\omega r$ obeys
the relativistic law of the composition of velocities, but as a
consequence, $\omega$ is no longer additive, whereas in the FTT
approach, the linear additivity of $\omega$ and the relativistic
law of composition of velocities are both preserved, thanks to the
relation $v=c\tanh (\omega r/c)$.

Trocheris' 1949 derivation is probably the most convenient to recall here. He considers ``as infinitesimal transformation in the neighbourhood of a
point $P$ the Lorentz transformation with velocity $\vecteur{\omega} \times
\vecteur{r}$'', where $\vecteur{\omega}$ is parallel to the $z$ axis. It is
\[
\begin{array}{rl}
dr' = dr, & d\theta' = d\theta - \alpha\, dx^{0},\\
dz' = dz, & dx^{0'} = dx^{0} -\alpha r\,d\theta,
\quad \text{and } \alpha = \omega r/c.
\end{array}
\]
{Since} $dx^{0'}$ is not a perfect differential, he adds {ad hoc} to it
$-2\alpha \theta\, dr$, in order to make the infinitesimal transformation
integrable. The generator of the new transformation is then
\begin{equation}
X^{a} = (-r^{2}\theta /c, 0, -x^{0} /c, 0)
\end{equation}
{Integrating} it, one obtains the relativistic rotation transformation (RRT):
\begin{equation}\label{eq:RRT}
\left\{
\begin{array}{rl}
r' = r, & \theta' = \theta \cosh \alpha - (x^0/r) \sinh \alpha , \\
 z' = z, &
x^{0'} = x^{0} \cosh \alpha - r\theta \sinh\alpha
\text{ with, again, } \alpha = \omega r/c.
\end{array}
\right.
\end{equation}
the initial conditions being assumed to be
$\theta=t=0$, $r$ and $z$ being constant.
For $c\approx\infty$, the RRT reduces to the Galilean rotation transformation to leading order in $1/c$.

Franklin obtained more directly the RRT \eqref{eq:RRT} as a special Lorentz transformation in the variables $x^0$, $r\theta$. Note, however, that both derivations make sense only if \emph{$r$ and $z$ are held constant}. Therefore, the RRT should not be interpreted as a change of coordinates or as a diffeomorphism in Minkowski space, since it really involves two variables and two parameters, not four coordinates.
Since the RRT is determined at $r = \text{const.}$,
and $z = \text{const.}$, it is not a global transformation of $M_4$, and indeed, it is not a Lorentz transformation in $M_4$. We suggest that it defines the tetrad of the local rotating observer, which is part of an observer manifold. The spacetime of the observer is, therefore, flat, but cannot be identified with $M_4$. Intuitively, it is a four-space that touches $M_4$ at one point only. The precise treatment relies on the formulae derived next.

\subsection*{Relativistic Rotational Tetrads}

To achieve the correct interpretation of the RRT, we retain from it only the expression of the four-velocity of the rotating particle and associate with a rotating body the tetrad adapted to it, namely the one in which the vectors along $z$ and $r$ do not change, while the other two, corresponding to the $x^0$ and $r\theta$ directions, are modified as in the RRT. This yields the vectors $\EE_h$ ($h=0,\dots,3$) with the following components:
\begin{equation}\label{eq:Eh}
E_{h}^{a} = \{(\cosh\alpha , 0, \sinh\alpha /r, 0),
    (0, 1, 0, 0),
    (\sinh\alpha, 0, \cosh\alpha /r, 0),
    (0, 0, 0, 1)\}
\end{equation}
\begin{Remark} In the coordinates $(x^{0}, r, r\theta, z)$,
\[
\begin{array}{l}
X_{1}^{a} = (0, -rx^{0}, 0, -r\theta), \\
X_{2}^{a} = (\sinh\alpha , 0, \cosh\alpha , 0) = b^{a},\\
X_{3}^{a} = (\cosh\alpha , 0, \sinh\alpha , 0) = u^{a}
\end{array}
\]
determine a Lie algebra, but it is not that of the Lorentz group; its commutation relations are as~follows:
\[
[X_{1} , X_{2}] = -X_{3} ,\quad
[X_{2} , X_{3}] = 0,\quad
[X_{3} , X_{1}] = X_{2}.
\]
{It} is of Bianchi Type VI.
\end{Remark}
This suggests the following definition:
\begin{Definition}
Let $\mathscr C$ denote a trajectory $s\mapsto x^a(s)$ of a material particle in uniform rotation about an axis, with unit time-like tangent vector $u^a(s)$ in Minkowski space $\cale:=M_4$. Here, a unit vector $\ell^a$ is chosen. It generates the timeline of the laboratory observer in $M_4$. The observer manifold $\calm_R$ relevant to this situation consists of copies of $M_4$, each attached to a different point of the trajectory, together with the frame $(\EE_h)$ given by the formulae \eqref{eq:Eh} above. The transfer map from an observer to the event manifold is the boost sending $u^a$ to $\ell^a$. The traces of the observers on $\cale$ give the trajectory of the particle. In addition, there is a transfer map between the observers with proper times $s$ and $s'$, namely the (unique) Lorentz transformation sending the frame at $s$ to the frame at $s'$.
\end{Definition}
Writing Maxwell's equations in this frame restores the covariance of electrodynamics~\cite{K:aflb:96b}, whereas the Galilean transformation leads to paradoxes that ultimately express that the Galilean rotation transformation does not preserve the Minkowskian structure, whereas the transfer maps on $\calm_R$ do.

\section{Second Case Study: Thomas Precession and the Relativity of Simultaneity}
\label{sec:thomas}

As a relativistic effect yielding even at \emph{small} velocities the factor of two needed to account for multiplet structure in atomic spectra, Thomas precession was a great surprise for the theoretical physicist of the late 1920s, when Thomas~\cite{thomas-26,thomas-27} and, less convincingly, J. Frenkel~\cite{frenkel-26} discovered it, using spin--orbit coupling. Though the normal multiplet structure could be derived through the Dirac equation, the semi-classical treatment was revisited from time to time, since, on the one hand, it provides a physical picture that the quantum treatment is lacking and, on the other hand, the expectation value of the spin operator has the same time dependence as in its classical equation of motion in the non-relativistic case.

Usual explanations of this effect are of two types:
\begin{enumerate}
\item [(a)] The kinematical one, which derives it from the non-commutativity of boosts (Lorentz transformations without rotation), in coordinate representations or in the spinorial formulation.
\item [(b)] The dynamical one, which relies on the covariant generalisation of the equation of motion of a spinning electron in a homogeneous electromagnetic field, the Thomas precession appearing as a consequence of the orthogonality of the spin-vector and the four-velocity of the electron.
\end{enumerate}
{Here}, we trace this precession back to the Relativity of simultaneity, through the systematic use of tetrad bases, and shed, in particular, some light on the reciprocity principle, which has been met with some confusion in the literature.

\subsection{The Kinematical Approach}

This was initiated by L. H. Thomas, who clearly connected the effect to the non-\linebreak commutativity of boosts; this latter property has been shown to be reflected in the Lie algebraic structure of the Lorentz group by V. Lalan~\cite{lalan-37}, although H. Poincar\'e (\cite{poincare-rendiconti-1906}; see also~\cite{poincare-cr-1905}) had already noted it. Again, V. Lalan~\cite{lalan-53} and R. Garnier~\cite{garnier-53} connected the spatial rotation appearing in the product of boosts to parallel transport on the pseudosphere, without mentioning \'E. Borel's still earlier work~\cite{borel-14} on the Lobachevsky--Bolyai velocity space in Special Relativity (see also V. Fock~\cite{fock-59}); these ideas were reformulated in terms of ``physical holonomy'' by H. Urbantke~\cite{urbantke-90} and the spatial rotation computed using Clifford algebra. Iterating the infinitesimal Lorentz transformation, W. H. Furry~\cite{furry-55} derived the Thomas precession at an order next to that of the Relativity of simultaneity and lower than that corresponding to the composition of velocities. The kinematical approach has also been discussed in the spinorial formalism, especially by C. Misner, K. S. Thorne, and J. A. Wheeler~(\cite{MTW}, pp.~175--176, 1118, 1146) and, more recently, by D. Hestenes~\cite{hestenes-74} and D. Hamilton~\cite{hamilton-81}, who states, ``that the usual interpretation is untenable'' and advocates the use of Clifford algebra. Textbooks discussing the kinematical derivation are many; those of C. M{\o}ller (\cite{moller}, pp.~117--125), J. D. Jackson (\cite{jackson}, 11.8 and 11.11), and Misner, Thorne, and Wheeler~\cite{MTW} are informative.

\subsection{The Dynamical Approach}

The motion of a spinning particle in a homogeneous electromagnetic field, first considered by J. Frenkel and discussed by W. A. Kramers (\cite{kramers-qm}, p.~226), was formulated by V. Bargmann, L. Michel, and V. L. Telegdi (BMT)~\cite{BMT}. They derived from the assumed orthogonality of the spin-vector and the four-velocity of the particle the known propagation of the spin-vector (recognised earlier as the Fermi--Walker propagation by F.A.E. Pirani~\cite{pirani:55}, in the case of a spinning particle in a gravitational field). That Thomas' and BMT's approach lead to the same result was confirmed by H. Bacry~\cite{bacry:62}, introducing, however, an additional postulate: ``a moving observer can measure directly the spin-vector in the rest-frame of the particle''. This, of course, is not satisfactory since this is tantamount to assuming at the outset that any inertial observer is equivalent to the observer travelling with the particle.

\subsection{The Main Argument}

The two ways, (a) and (b) above, of accounting for Thomas precession can be traced back to the Relativity of simultaneity, through the systematic use of tetrad bases. Indeed, already at the lowest order in $\beta=v/c$, hypersurfaces of simultaneity are not invariant under a boost, and space axes of two frames related by it do not remain parallel in Minkowski spacetime; thus, they should be rotated after a series of boosts, the product of which leaves the initial time axis invariant. On the other hand, the infinitesimal boost mapping the tangent to the world-line of a particle to the neighbouring one induces the Fermi--Walker transport of space axes. It, then, follows that the above explanations of Thomas precession are simply direct consequences of the lack of a notion of simultaneity common to all observers on the observer manifold. We, therefore, introduce an observer manifold $\calm$ adapted to this situation. The event manifold $\cale$ is Minkowski space.

\begin{Definition}
Let $\mathscr C$ denote a trajectory $s\mapsto x^a(s)$ of a material particle, with unit time-like tangent vector $u^a(s)$ in $\cale:=M_4$. The observer manifold $\calm_T$ relevant to Thomas precession consists of copies of $M_4$, each attached to a different point of the trajectory. The traces of the observers on $\cale$ give the trajectory of the particle. The transfer map from an observer to the event manifold is the boost sending $u^a$ to a fixed time-like unit vector $\ell^a$. Here, again, there is also a transfer map between the observers with proper times $s$ and $s'$, namely the boost sending $u^a(s)$ to $u^a(s')$.
\end{Definition}

\subsection{Infinitesimal Transformations on the Observer Manifold and Fermi--Walker Transport}

We now consider the infinitesimal transformation of the transfer map. It is this form that will yield Thomas precession. We show that it leads to the Fermi--Walker transport of the space triad.
\begin{Theorem}
A vector $p^i=p^a e_a{}^i(s)$ moving along the trajectory $\mathscr C$ (or \emph{comoving vector} for short), with the $p^a$ constant, evolves according to
\begin{equation}\label{th-22}
 \dot p^i = (u^i\dot u_j-u_j\dot u^i)p^j.
\end{equation}
{In} other words, $p^i$ is Fermi--Walker-propagated along the curve tangent to $u(s)$.
\end{Theorem}
\begin{proof} Quite generally, the boost $B$ computed earlier determines an infinitesimal boost $b$ when $v=u+\delta u$, with $u\cdot\delta u=0$ (which ensures $v\cdot v=-1$); we have
\begin{equation}\label{th-17}
 b^i{}_j = \delta^i{}_j+(u^i\delta u_j-u_j\delta u^i),
\end{equation}
and the $b$-transforms $E_a{}^i$ of $e_a{}^i$ are
\begin{equation}\label{th-18}
 E_a{}^i = e_a{}^i + (u^i\delta u_j-u_j\delta u^i)e_a{}^j.
\end{equation}

When the four-velocity increment $\delta u$ is equated with $\dot u ds$, we obtain
\begin{equation}\label{th-21}
 \delta e_a{}^i = \dot e_a{}^i ds =(u^i\dot u_j-u_j\dot u^i)ds,
\end{equation}
so that, by the boost sending $u$ to $u+\dot u ds$, the space axes $(e_\lambda{}^i)$ are Fermi--Walker-propagated along the world-line of the origin of the moving frame with variable velocity $u^i(s)$.

Now, a comoving vector $p^i=p^a e_a{}^i$ is transformed into $p^i+\dot p^ids$, so that
\begin{equation}
 \dot p^i = (u^i\dot u_j-u_j\dot u^i)p^j,
\end{equation}
in other words, $p^i$ is Fermi--Walker-propagated along the line tangent to $u(s)$, as desired.

Let $B+\dot B ds$ be the boost sending $u+\dot u ds$ to $v^i=B^i{}_ku^k$; we have
\begin{eqnarray}
\dot v^i &=& B^i{}_k\dot u^k+\dot B^i{}_ k u^k=0, \label{th-23a}\\
\dot E_\lambda{}^i &=& \Omega^i{}_ jE^j_\lambda, \\
\text{ where } \Omega^i{}_ j &=&
 \left[B^i{}_ k(u^k\dot u_h-u_h\dot u^k)+\dot B^i{}_h
 \right] (B^{-1})^h{}_ j, \label{th-23c}
\end{eqnarray}
so that $\Omega^i{}_ k v^k=0$, by virtue of Equation~(\ref{th-23a}). After substituting the values of $B$, $\dot B$, and $B^{-1}$ into Equation~(\ref{th-23c}), we obtain
\begin{equation}\label{th-24}
 \Omega^i{}_ j =
 \left[\dot\gamma(v^i u_j-v_j u^i)+(u^i\dot u_j-u_j\dot u^i)
   -\gamma(v^i\dot u_j-v_j\dot u^i)
 \right](1+\gamma)^{-1}.
\end{equation}
The $B$-transforms $P^i$ of $p^i$ satisfy
\begin{equation}\label{th-25a}
\dot{\tilde P}^i = p^\rho\dot E_\rho{}^i=\Omega^i{}_k\tilde P^k,
\text{ where }\tilde P^i=(\delta^i{}_ k+v^iv_k)P^k,
\end{equation}
or
\begin{equation}\label{25b}
 \dot P^\lambda = \Omega^\lambda{}_\sigma P^\sigma,
\end{equation}
with
\begin{equation}\label{th-25c}
 \Omega^\rho{}_ \sigma=\Omega^i{}_j E_i{}^\rho E_\sigma{}^j
    =\frac{\gamma^2}{1+\gamma}(\beta^\rho\Gamma_\sigma-\beta_\sigma\Gamma^\rho),
\end{equation}
where we have used the relation $u^iE_ i{}^\rho=\gamma\beta^iE_i^\rho$ and have set $\Gamma^\rho := \dot\beta^iE_i{}^\rho$.

It follows that any comoving vector $p$, \emph{whether orthogonal to $u$ or not}, is Fermi--Walker-propagated along the $u$-line, and its image by the boost $B$ has its spatial part rotated by $\Omega^\rho{}_\sigma$ with respect to $(E_\lambda)$. This property is, therefore, not restricted to the spin-vector, as it would appear from the BMT approach.
\end{proof}

\subsection{Covariant Form of the Equations}

Consider the motion of a charged spinning electron in a homogeneous electromagnetic field $F^{ij}$. The Bargmann--Wigner polarisation vector $s^i$ attached to it is orthogonal to its four-velocity $u$:
\begin{equation}\label{th-26}
 u\cdot s:=u^as_a=0,
\end{equation}
and determines its magnetic moment:
\begin{equation}\label{th-27}
 w^i = (ge/2m) s^i,
\end{equation}
where $g$ is the Land\'e factor. With respect to $(e_a)$, the Larmor precession is given by
\begin{equation}\label{th-28}
 \dot s_L{}^i = - \tilde F^{ij}w_j,\text{ or }
 \dot s_L{}^\lambda = - \tilde F^{\lambda\sigma}w_\sigma,
\end{equation}
where $\tilde F^{ij}=(\delta^i{}_k+u^iu_k)(\delta^j{}_h+u^ju_h)F^{kh}$. On the other hand, the accelerated motion of the electron due to the Lorentz force is given by
\begin{equation}\label{th-29}
 \dot u^i = -(e/m)F^i{}_ku^k
\end{equation}
and according to Equation~(\ref{th-22}), the spin-vector is Fermi--Walker-propagated so that we obtain a contribution to the precession equal to
\begin{equation}\label{th-30}
 \dot s_{FW}{}^i=-(e/m)(u^iF_k{}^h-F^{ih}u_k)u_hs^k.
\end{equation}
It follows that the total precession is
\begin{equation}\label{th-31}
 \dot s^i = -(e/2m)\left[gF^{ik}+(g-2)F^{hk}u^iu_h\right]s_k.
\end{equation}
Equation~(\ref{th-31}) is the BMT equation where the additional term responsible for the Thomas precession is now simply accounted for by the Fermi--Walker propagation of comoving space axes during the motion of the electron; it is not a consequence of the orthogonality relation~(\ref{th-26}).

\subsection{Evaluation in the Atom's Rest Frame}

 In applications to atomic spectra, one is actually interested in evaluating the precession in the rest frame of the atomic nucleus (\cite{eisberg-67}, p.~340), which we take as frame $(E_a)$. With respect to $(E_a)$, the motion of the electron is given by
 \begin{equation}\label{th-32}
 \dot u^iE_i{}^\lambda = -(e/m)(F^{\lambda 0}+F^{\lambda\sigma}\beta_\sigma)
 \end{equation}
where $F^{ab}=F^{ij}E_i{}^aE_j{}^b$ and $u^i=\gamma(E_0{}^i+\beta^\sigma E_\sigma{}^i)$.

In terms of $E^\lambda=F^{\lambda 0}$ and $B^\lambda=\frac12\ep^{\lambda\rho\sigma}F^{\rho\sigma}$, Thomas precession, given by Equation~(\ref{th-25c}), becomes
\begin{equation}\label{th-32bis}
\Omega^\lambda{}_\rho=-\frac{\gamma^2}{1+\gamma}
\ep^\lambda{}_{\rho\sigma}
\left[(\beta\times E)^\sigma+(\beta\times(\beta\times B))^\sigma\right],
\end{equation}
where $(a\times b)^\lambda=\ep^{\lambda\rho\sigma}a_\rho b_\sigma$ is the usual vector product.

Hence, the total precession is
\begin{eqnarray}
\omega^\lambda{}^\rho & = & -(e/2m)\ep^\lambda{}_{\rho\sigma}
\left[gB^\sigma-(\beta\times E)^\sigma
\right]/(1+\gamma) \nonumber\\
 & & \mbox{}+(g-2)(e/2m)\ep^\lambda{}_{\rho\sigma}
\left[\frac{\gamma^2}{1+\gamma}(\beta\times(\beta\times B))^\sigma+\gamma(\beta\times E)^\sigma
\right],
\label{th-33}
\end{eqnarray}
as in the usual approaches.

\begin{Remark} For the dotted derivative $\dot Q$ of a quantity $Q$ being defined as $u^i\pa_i Q$, the time derivative in the rest frame of the atom is obtained by dividing through by $\gamma$: $(d/dt)Q=\dot Q/\gamma$.
\end{Remark}

\begin{Remark} If Cartesian coordinates in $M_4$ are adopted, with $e_a{} ^i=\delta_a{}^i$, we recover the usual expression for the Thomas precession.
\end{Remark}

\section{Third Case Study: Charged Particle in a Constant Electromagnetic Field}\label{sec6}

We revisit the motion of a charged particle in an electromagnetic field and show how to associate with it an observer manifold. In this case, the particle dynamics determine a special frame attached to the particle, which naturally defines a splitting of the tangent space into orthogonal subspaces a two-dimensional Minkowski space and a two-dimensional space. The observer manifold is here determined by the dynamics. Let $\mathbf{u}=u^a\partial_a$, with $u^a=dx^a/ds$, be the time-like unit vector tangent to the world-line $(L)$ of a particle in the
spacetime $M_4$ endowed with the metric $\eta_{ab}$. We relied on an earlier paper~\cite{K:aflb:03}, to which the reader is referred for background information; in particular, the Frenet frame is closely related to the Franklin rotation transformation. The notation used here also follows this paper.

The dynamics generate a Frenet--Serret tetrad
$\{\mathbf{u},\mathbf{n},\mathbf{b},\mathbf{c}\}$ determined by the integration of the ordinary differential equations:
\begin{equation}\label{eq:FS-s}
 \dot{\mathbf{u}} = a\mathbf{n},\quad
 \dot{\mathbf{n}} = a\mathbf{u} + \tau \mathbf{b},\quad
 \dot{\mathbf{b}} = -\tau \mathbf{n} + \sigma \mathbf{c},\quad
 \dot{\mathbf{c}} = -\sigma \mathbf{b},
\end{equation}
where the dot denotes $d/ds$; the components of $\mathbf{u}$ are $u^a$ and similarly for the others. These four-vectors are mutually
orthogonal and form a right-handed frame; we have
\[
u^au_a=-1,\quad n^an_a=b^ab_a=c^ac_a=1.
\]
{Here}, $a$, $\tau$ and $\sigma$ are, respectively, the curvature,
torsion, and third curvature of $(L)$. The curvature and torsion
are nonnegative; we let $\ep$ denote the sign of $\sigma$ (equal
to $-1$ if \mbox{$\sigma<0$}, $+1$ otherwise). The Frenet frame may be
thought of as the result of Gram--Schmidt orthonormalisation
applied to the first four derivatives (with respect to $s$)
of the position~${\mathbf x}$.

Since the electromagnetic field is constant, so are the generalised curvatures $a$, $\tau$, and $c$. One might, therefore, feel that it suffices to compute the exponential of a matrix to solve the problem entirely~\cite{K:aflb:03,synge,scarr-friedman-acc}; incidentally, since some authors view such particles as having uniform acceleration~\cite{scarr-friedman-acc}, we must point out that this is debatable~\cite{K-cr-65b}, even though a full discussion would take us too far. Coming back to our main concern, while the solutions may be computed in terms of the arc length parameter $s$ by investigating the eigenvalues of the matrix $S$ below, the problem has additional structure: the Frenet frame determines a distinguished tetrad in which the last two vectors are decorrelated from the first two. It follows that the problem generates not only a frame, but a distinguished splitting of Minkowski space into two planes. We shall use this to obtain a new example of an observer manifold by taking this tetrad as defining the local observer.

The eigenvalues of the matrix
\[S:=(S_A{}^B)=\left[
\begin{matrix}
\quad 0\quad & a & 0 & 0 \\
a & \quad 0\quad & \tau & 0 \\
0 & -\tau & \quad 0\quad & \sigma \\
0 & 0 & -\sigma & \quad 0\quad
\end{matrix}
\right]
\]
solve
\begin{equation}\label{eq:roots-s}
 (\lambda^2-a^2)(\lambda^2+\tau^2+\sigma^2)+(a\tau)^2=0.
\end{equation}
{This} equation has four roots: $\pm\chi$, $\pm i \omega$, where the
nonnegative real numbers $\chi$ and $\omega$ are given by
\begin{eqnarray}
\omega^2 &=& -\frac12[a^2-\tau^2-\sigma^2+\Delta], \label{eq:roots-1}\\
\chi^2 &=& -\frac12[a^2-\tau^2-\sigma^2-\Delta], \text{ with } \label{eq:roots-2}\\
\Delta &=& \sqrt{(a^2-\tau^2-\sigma^2)^2+4a^2\sigma^2}. \label{eq:roots-3}
\end{eqnarray}
{The} relations on the sum and
product of roots of the equation yield
\begin{equation}\label{eq:roots-s1}
\chi^2-\omega^2 = a^2-\tau^2-\sigma^2;\quad
a\sigma=\ep\omega\chi.
\end{equation}
{Note} also that
\[ \chi^2+\omega^2=\Delta.
\]
{From} Equation \eqref{eq:roots-s}, with $\lambda=\chi$, it is apparent that $\chi^2\leq a^2$. Equation \eqref{eq:roots-s1} then yields $\omega^2=\tau^2+\sigma^2+\chi^2-a^2$; hence,
$\omega^2\leq\tau^2+\sigma^2$. It, therefore, makes sense to define
two positive numbers $\Gamma$ and $\Lambda$ by
\begin{equation}
 \Gamma^2=\frac{a^2+\omega^2}{\chi^2+\omega^2};\quad
 \Lambda^2=\frac{a^2-\chi^2}{\chi^2+\omega^2},
\end{equation}
so that
\[ \Gamma^2-\Lambda^2=1.
\]
{Using} these relations, we obtain
$(\Gamma\chi)^2+(\Gamma^2-1)\omega^2=a^2$; hence,
\begin{equation} a^2=\Gamma^2\chi^2+\Lambda^2\omega^2.
\end{equation}
{Also},
$(\Gamma\Lambda)^2(\chi^2+\omega^2)^2=(a^2+\omega^2)(a^2-\chi^2)
=(a^2+\omega^2)(\tau^2+\sigma^2-\omega^2)=(a\tau)^2$. (Equation~(\ref{eq:roots-s}) has
been used to obtain the last equality.) Therefore,
\begin{equation}
a\tau=\Gamma\Lambda(\chi^2+\omega^2).
\end{equation}
{The} following relations are also useful.
\begin{eqnarray}
(\chi\Gamma)^2(\chi^2+\omega^2) &=&
\chi^2(\omega^2+a^2)=a^2(\sigma^2+\chi^2) \\
(\omega\Lambda)^2(\chi^2+\omega^2) &=&
\omega^2(a^2-\chi^2)=a^2(\omega^2-\sigma^2).
\end{eqnarray}

We assume now that $a$, $\tau$, and $\sigma$ are constant. We prove the following.
\begin{Theorem}\label{th:frame}
The vectors:
\begin{eqnarray}
f_0 &=& \Gamma\uu+\Lambda\bb; \\
f_1 &=& \Lambda\uu+\Gamma\bb; \\
f_2 &=& a^{-1}[\Gamma\chi\nn+\ep\Lambda\omega\cc]; \\
f_3 &=& a^{-1}[-\ep\Lambda\omega\nn+\Gamma\chi\cc]
\end{eqnarray}
form a right-handed orthonormal frame such that
\begin{eqnarray}
\ddot f_0-\chi^2 f_0 &=& \ddot f_1+\omega^2 f_1 = 0;\\
\ddot f_2-\chi^2 f_2 &=& \ddot f_3+\omega^2 f_3 = 0.
\end{eqnarray}
\end{Theorem}
This theorem is proven by direct computation of the derivatives. We note
that the last two relations follow from the first, because one can
check directly that
\[ \dot f_0=\chi f_2;\quad \dot f_1=\ep\omega f_3.
\]
{The} vector $f_0$ has the following interpretation. Let
\[ y^a = x^a +a \omega^{-2}n^a.
\]
{The,}
\[ \frac{d\mathbf y}{ds} = \frac {a\tau}{\Lambda\omega^2}f_0.
\]
{The} vectors $f_k$ may also be obtained differently: one has
\begin{eqnarray}
\ddot \uu +\omega^2 \uu &=& \frac{a\tau}\Lambda f_0;\\
\ddot \uu -\chi^2 \uu &=& \frac{a\tau}\Gamma f_1;\\
\ddot \nn +\omega^2 \nn &=& \frac{\Gamma\chi}a (\chi^2+\omega^2) f_2
=\chi\frac\tau\Lambda f_2;\\
\ddot \nn -\chi^2 \nn &=& \ep\frac{\Lambda\omega}a (\chi^2+\omega^2) f_3
=\ep\omega\frac\tau\Gamma f_3.
\end{eqnarray}
The properties of the frame $(f_k)$ suggest the following definition.
\begin{Definition}\label{def:obsmanif-c}
Let $\mathscr C$ denote a trajectory $s\mapsto x^a(s)$ of a material particle in an electromagnetic field, with unit time-like tangent vector $u^a(s)$ in $\cale:=M_4$. The observer manifold $\calm_R$ relevant to this situation consists of copies of $M_4$, each attached to a different point of the trajectory, together with the frame given $(f_0,f_1,f_2,f_3)$ defined by the formulae of Theorem~\ref{th:frame}. The traces of the observers on $\cale$ give the trajectory of the particle. The transfer map from an observer to the event manifold is the boost sending $u^a$ to a fixed time-like unit vector $\ell^a$. The transfer map from the observers with proper times $s$ and $s'$ is the (unique) Lorentz transformation sending the frame at $s$ to the frame at $s'$.
\end{Definition}

\section{Electromagnetic Field Components}\label{sec7}

While we have focused on a constant field so far, formulae involving the Frenet--Serret tetrad and, therefore, the observer frame of Definition~\ref{def:obsmanif-c}, in terms of the field derivatives with respect to the arc length, using the formulae:
\[ \dot F_{ab} = u^c\nabla_c F_{ab}
\]
and
\[ \ddot F_{ab}=u^du^c\nabla_d\nabla_c F_{ab}+an^c\nabla_c F_{ab}.
\]
{Higher} derivatives could be computed in the same way.
We give here explicit formulae for the electric and magnetic field components, defined by
\begin{align*}
E^a &= F^{ab}u_b\\
B^a &= -\overset{*}{F}{}_{ab}u^b, \text{ with}\\
\overset{*}{F}{}_{ab} &=\frac 12 \eta_{abcd}F^{cd},
\end{align*}
where, as usual, $\eta_{abcd}=\sqrt{-g}\ep_{abcd}$ (with, here, $-g=1$).
In other words, for every index $n$,
\[ B_n = -\frac12\eta_{npqr}u^pF^{qr};
\]
hence,
\begin{align*}
\eta^{abmn}u_mB_n &=\frac12\ep^{abmn}u_m\ep_{npqr}u^pF^{qr}\\
&= -\frac12\delta^{abm}_{pqr}u_mu^pF^{qr}\\
&= F^{ab}-(u^aF^{bm}-u^bF^{am})u_m.
\end{align*}
{It} follows that $F^{ab}$ may, in turn, be recovered from its electric and magnetic parts.
\begin{equation}
F^{ab}=\eta^{abmn}u_mB_n+u^aE^b-u^bE^a.
\end{equation}
{In} the sequel, $\EE$ has components $E^a$, and similarly for $\mathbf B$, and expressions such as $\nn\ddot\FF\uu$ are short for contractions such as $n^a\ddot F_{ab}u^b$.
\begin{Theorem}
In any e-m field, if $a\tau\neq 0$,
\[ k\EE = a\nn\text{ and }k\mathbf B=B_1\nn+B_2\bb+B_3\cc
\]
with $k=Q/mc$,
\begin{eqnarray*}
B_1 &=& \frac\sigma k-\frac 1{a\tau}\left[
2a\cc\dot\FF\nn+\cc\ddot\FF\uu-\dot a B_2\right],
\\
B_2 &=& \frac 1a\cc\dot\FF\uu,
\end{eqnarray*}
and
\[ B_3 = \frac\tau k-\frac 1a \bb\ddot\FF\uu.
\]
{In} a constant field, $B_2=0$, $B_3 = \tau/k$, and $B_1 = \sigma/k$.
\end{Theorem}
\begin{proof}
Let us write
\[ \BB=B_0\uu+B_1\nn+B_2\bb+B_3\cc.
\]
{First} of all, $\BB$, like $\EE$, is always orthogonal to $\uu$.
Therefore, $B_0=0$.
Next, from the expression of $F^{ab}$ in terms of the fields $\EE$ and
$\BB$, we have, using the definition $\eta^{abmn}=-\frac 1{\sqrt{-g}}\ep^{abmn}$ and the fact that the Frenet--Serret tetrad is a right-handed frame by the
choice of $\cc$ (this is why $\sigma$ has a sign),
\begin{align*}
c^aF_{ab}b^b &= \eta^{abmn}c_ab_bu_mB_n=B_1\\
n^aF_{ab}c^b &= \eta^{abmn}n_ac_bu_mB_n=B_2\\
b^aF_{ab}n^b &= \eta^{abmn}b_an_bu_mB_n=B_3.
\end{align*}
{By} repeated differentiations, we obtain
\begin{eqnarray*}
\dot\uu &=& k\FF\uu=a\nn
\\
\ddot\uu &=& k(a\FF\nn+\dot\FF\uu)
\\
&=& a(a\uu+\tau \bb)+\dot a\nn
\\
\dddot\uu &=& k\left[ \FF(\dot a\nn+a(a\uu+\tau \bb))+2\dot\FF(a\nn)+\ddot\FF\uu
\right]
\\
&=& (d/ds)\left[a^2\uu+a\tau\bb+\dot a\nn\right]
\\
&=& a^3\nn+2a\dot a\uu+a\tau(-\tau\nn+\sigma\cc)+d(a\tau)/ds\bb+\ddot a\nn+\dot
a(a\uu+\tau\bb)
\\
&=& 3a\dot a\uu+(\ddot a+a^3-a\tau^2)\nn+(2\dot a\tau+a\dot\tau)\bb+a\tau\sigma\cc.
\end{eqnarray*}
{Projecting} on the FS tetrad, we obtain
\[ k\nn\FF\uu=a;\quad \uu\FF\uu=\bb\FF\uu=\cc\FF\uu=0.
\]
{The} relations involving $\ddot\uu$ yield, in addition to the
known relation $\nn\FF\uu=a/k$,
\begin{eqnarray*}
\nn\dot\FF\uu &=& \dot a/k \\
a\bb\FF\nn+\bb\dot\FF\uu &=& a\tau/k \\
a \cc\FF\nn+\cc\dot\FF\uu &=& 0.
\end{eqnarray*}
{The} second of these equations reads
\[ aB_3+\bb\dot\FF\uu = a\tau/k,
\]
hence the value of $B_3$. Similarly, the third of these relations
yields
\[ aB_2=\cc\dot\FF\uu.
\]
{Thus}, $B_2$ and $B_3$ may be determined from the field components on the Frenet--Serret tetrad.

To determine $B_1$, we considered the relations derived by the projection of the third
derivative; the projection on $\uu$ yields the same information as the projection of the
second derivative; taking the scalar product with the other three tetrad vectors gives:
\begin{eqnarray*}
k\{ a^2\nn\FF\uu+a\tau\nn\FF\bb+\nn\ddot\FF\uu\}
&=& \ddot a+a(a^2-\tau^2) \\
k\{\dot a\bb\FF\nn+a^2\bb\FF\uu+2a\bb\dot\FF\nn+\bb\ddot\FF\uu\}
 &=& 2\dot a\tau+a\dot\tau \\
k\{\dot a\cc\FF\nn+a^2\cc\FF\uu+a\tau\cc\FF\bb+2a\cc\dot\FF\nn+\cc\ddot\FF\uu\}
&=& a\tau\sigma.
\end{eqnarray*}
{These} equations simplify to
\begin{eqnarray*}
-a\tau B_3+\nn\ddot\FF\uu &=& [\ddot a-a\tau^2]/k \\
\dot aB_3+2a\bb\dot\FF\nn+\bb\ddot\FF\uu
&=& [2\dot a\tau+a\dot\tau]/k \\
-\dot aB_2+a\tau B_1+2a\cc\dot\FF\nn+\cc\ddot\FF\uu
&=& a\tau\sigma/k.
\end{eqnarray*}
{Rearranging}, this may also be written as
\begin{eqnarray*}
\nn\ddot\FF\uu &=& \ddot a/k+{a\tau}(B_3-\tau/k) \\
2a\bb\dot\FF\nn+\bb\ddot\FF\uu
&=& [\dot a\tau+a\dot\tau]/k+\dot a(\tau/k-B_3) \\
2a\cc\dot\FF\nn+\cc\ddot\FF\uu
&=& {a\tau}(\sigma/k-B_1)+\dot a B_2.
\end{eqnarray*}
{The} last of these relations yields $B_1$, since $B_2$ is already known.
\end{proof}

\section{Conclusions}\label{sec8}

The conception of rotational motion as a parameterised set of motions is strongly rooted in absolute simultaneity. Three alternatives have been proposed, corresponding to somewhat different setups, leading to rotational or more-general accelerated motions. In all cases, the event manifold $\cale$ is Minkowski space $M_4$.
The interpretation of rotation transformations as global transformations from Minkowski space to a single, possibly curved manifold leads to insuperable difficulties. We, therefore, suggested that the local observer at every point of the trajectory of an object in uniform motion is associated with a different manifold, the collection of which forms an observer manifold. In the present examples, the local manifold of an observer is a pseudo-inertial system in the sense of C-equivalence.

Axioms for the observer manifold were compared and were shown to include, as special cases, manifolds in the usual sense and frame and fibre bundles over the event manifold. They are not fibre bundles, because not every event has a fibre over it, and one point may have several fibres if several observers in different states of motion happen to meet at a single event. For this reason, fairly general axiomatics were proposed, so as to separate clearly the mathematical issue of the classification of observer manifolds from their physical interpretation in special cases.

The first example is Franklin's rotation transformation, rediscovered by Trocheris and, later, by Takeno. It should be understood as defining a parameterised family of tetrads in cylindrical coordinates, not as a diffeomorphism, let alone a change of coordinates. Its usefulness in pulsar physics has already been established. The observer manifold consists of a collection of Minkowski spaces, one for each point of the trajectory, defined by distinguished tetrads.

The second example is the motion of an electron with spin. Here, the transfer maps are Lorentz boosts, and the observers are determined by their time-like tangents. Thomas precession is a consequence of the fact that the observer manifold cannot be reduced to the event manifold because the transfer maps do not form a group. Indeed, the composition of two Lorentz boosts is not a boost in general. Fermi--Walker transport is recovered as the infinitesimal form of transfer maps between observers.

The third example is the motion of a charged particle in a constant electromagnetic field. While this also leads to the Franklin transformation in the case of motion in a plane, under the action of a pure magnetic field orthogonal to it, it leads in general to a more-complicated, but still completely explicit splitting of Minkowski space, adapted to the geometric structure of the Frenet--Serret tetrad. The transfer maps are Lorentz transformations between these tetrads.

Thus, the observer manifold is not only made necessary by the difficulties in the representation of uniform rotation, it provides a flexible mathematical framework for the interpretation of measurements by different observers, in a form both suitable for calculation and immediately consistent with the physical requirements of Relativity. It seems to provide a mathematical formulation for the views of Einstein.

Perspectives include (i) a discussion of further examples in which the manifolds attached to different observers are curved; (ii) the exploration of further axiomatics; (iii)~the classification of observer manifolds, including the possible transport rules obtainable as infinitesimal versions of transfer maps; (iv) the extension to variable electromagnetic fields and non-electromagnetic fields, including gravitational fields. For the latter, the introduction of a conformal factor, in the spirit of C-equivalence, is the obvious procedure. Thus, the wider mathematical issue is the classification of observer manifolds on the basis of the structure of their transfer maps and the comparison of different possible axiomatics, for which this study provided a first set of examples.

\textbf{Acknowledgement.} The author thanks Jan Peng at MDPI for suggesting to prepare and submit this paper, and the referees for their kind comments.

\end{document}